\documentclass{article}
\usepackage{graphicx}
\usepackage{amsmath}
\usepackage{wasysym}
\usepackage{amssymb}
\usepackage{amsthm}
\usepackage{commath}
\usepackage{mathtools}
\usepackage{regexpatch,fancyvrb,xparse}
\usepackage{tikz-cd}

\newtheorem{theorem}{Theorem}[section]
\newtheorem{lemma}[theorem]{Lemma}
\newtheorem{corollary}[theorem]{Corollary}
\newtheorem{example}[theorem]{Example}

\theoremstyle{definition}
\newtheorem{definition}[theorem]{Definition}

\theoremstyle{remark}
\newtheorem{remark}[theorem]{Remark}

\newcommand{\Z}{\mathbb Z}

\newcommand{\C}{\mathbb C}

\newcommand{\p}{$\phi$}
\newcommand{\la}{$\lambda$}

\newcommand{\G}{$G$}

\newcommand{\A}{$\mathcal A$}

\def\Sha{\hbox{$\amalg\kern-.39em\amalg$}}

\renewcommand{\bar}{\overline}
\renewcommand{\tilde}{\widetilde}

\begin{document}

\title{Twisted Orthogonality Relations For  Certain
$\Z/m\Z$ -graded  Algebras}

\author{Prashant Arote}
\date{}
\maketitle

\begin{abstract}

In this paper we will study the notion of a Frobenius $\star$-algebra and prove some orthogonality relations for the irreducible characters of a Frobenius $\star$-algebra.
Then we will study  $\Z/m\Z$-graded Frobenius $\star$-algebras and  prove some twisted orthogonality relations for them. 
\end{abstract}
\section*{Introduction}
The goal of this paper is to define twisted characters for a special class of $\Z/m\Z$-graded Frobenius algebras and prove some orthogonality relations for them.

We begin by recalling the notion of a Frobenius algebra and a symmetric Frobenius algebra.
In this paper we will only deal with  symmetric Frobenius algebras.
Further we will give a description of the centre of a symmetric Frobenius algebra and recall some orthogonality relations for their irreducible characters. 
All of these results are found in \cite{frobeniusalgebra}.

In  section 2 we define a $\textit{Frobenius $\star$-algebra}$. 
We will prove that Frobenius $\star$-algebras are semisimple and derive some orthogonality relations for their irreducible characters.
An example of a Frobenius $\star$-algebra is the group ring $\C[G]$ of a finite group $G$. 
Another example is that of a (complexified) based ring, which is discussed in \cite{BasedRings}.

In section 3 we will recall the notion of group graded rings and group graded modules over a group graded ring.
Further we state some results about Clifford theory for group graded rings without proof. 
We refer to \cite{Cliffordtheory} for more details.

In section 4 we define the notion of a twisted character for a $\Z/m\Z$-graded  Frobenius $\star$-algebra and by applying the results from sections 2 and 3 we obtain some orthogonality relations for twisted characters of a $\Z/m\Z$-graded symmetric Frobenius $\star$-algebra.

Let us describe an example of twisted characters and orthogonality relations. 
Let $G$ be a finite group and $\sigma$ be a automorphism of $G$ of order $m$. 
Therefore we get an action of $\Z/m\Z$ on $G$ and on Irr($G$) (= set of all irreducible representations of G over $\C$ upto isomorphism).
Then $\C[G\rtimes\Z/m\Z]$ is a $\Z/m\Z$-graded Frobenius $\star$-algebra. If $\rho \in \mbox{Irr}(G)^{\Z/m\Z}$, then we can extend $\rho$ to a representation $\tilde{\rho}$ of $G\rtimes\Z/m\Z$.
In this case the twisted character $\tilde{\chi}_{\rho}$ of $G$ corresponding to $\rho$
is  $\tilde{\chi}_{\rho}(g) := \chi_{\tilde{\rho}}(g,\sigma)$.
Then the results of this paper imply that the set of twisted characters $\{\tilde{\chi}_{\rho}: \rho \in \mbox{Irr}(G)^{\Z/m\Z}\}$ forms an orthogonal basis for the space of $\sigma$-twisted class functions on $G$.

The twisted orthogonality relations are used by G. Lusztig in his theory of character sheaves for reductive groups over a finite field.
Orthogonality relations for twisted characters of the Grothendieck algebras of $\Z/m\Z$-graded categories was proved by T. Deshpande. For more details see \cite{ta}.

\section{\textbf{Frobenius Algebra}}

Now we begin by recalling the definition of a Frobenius algebra and some of its properties.
\begin{definition} 
A Frobenius algebra \A\ is a finite dimensional, unital, associative $\C$-algebra equipped with a linear functional $\lambda:\mathcal{A}\rightarrow \C$ such that the bilinear form
on \A\ defined by $\left( a ,b \right)=\lambda (ab)$ is nondegenerate. 
 Moreover if $\lambda(ab)=\lambda(ba),\ \forall\ a,\ b\in \mathcal{A}$, then \A\ is called as symmetric Frobenius algebra.
\end{definition}
\noindent
$\bullet$ A linear functional $\lambda$ on an algebra \A\ is said to be a nondegenerate if the bilinear form on \A\ defined by $\left( a ,b \right)=\lambda (ab)$ is nondegenerate.\\
$\bullet$ To make connection more explicit, we say that the pair (\A,\la) is a Frobenius algebra, if \la\ is a nondegenerate linear functional on \A.

As \A \ is finite dimensional,  \A$^{*}$=Hom$_{\mathbb{C}}(\mathcal{A},\mathbb{C})$\ is an (\A, \A)-bimodule, where for $a\in$\A \ and $\gamma\in\mathcal{A}^{*}$\ we define
$$
\begin{cases}
( a.\gamma)(x)=\gamma (xa) \\
( \gamma.a)(x)=\gamma (ax)
\end{cases}
\forall\ x \in \mathcal{A}
$$

There is an equivalent definition of a Frobenius algebra: A Frobenius algebra \A\ is a finite-dimensional, unital, associative $\mathbb{C}$-algebra  equipped with a left \A -module isomorphism $\phi$:\A$\rightarrow\mathcal{A}^{*}$. 
In this case it is denoted by the pair (\A,\p). 

The equivalence between the two definitions is given by:
\begin{gather*}
(\mathcal{A},\lambda)\ \rightsquigarrow\ \phi(a)=\lambda(- \cdot a) \\
(\mathcal{A},\phi)\ \rightsquigarrow\ \lambda (x)=\phi (1)(x).
\end{gather*}
\noindent
$\bullet$  (\A,\p) is a symmetric Frobenius algebra if and only if \p\ is a (\A,\A) bimodule isomorphism between \A \ and \A$^{*}$. 

\noindent
$\bullet$ If $((a_{i}), (b_{i}))$ is an ordered pair of bases for \A \ such that \p$(b_{i})(a_{j})=\delta_{ij}$, then we say that the bases $(a_{i}), (b_{i})$ are $\mathbf{\phi}\textbf{-dual}$. If $(a_{i})$ is a basis of  (\A,\p), then there exists a second basis ($b_{i}$) of (\A,\p) such that $(a_{i})$, $(b_{i})$ are \p -dual. 
\begin{lemma}
Let (a$_{i}$), (b$_{i}$) be \p -dual bases for a Frobenius algebra (\A, \p), 
and let $\gamma$ $\in$ \A $^{*}$. 
Then
\[ 
\phi  ^{-1}(\gamma)=\sum_{i} \gamma(a_{i})b_{i}.
\]
\end{lemma}

\begin{definition}
 A linear functional $\gamma$ on an algebra \A\ is said to be a class functional if $\gamma(ab)=\gamma(ba)\ \forall\ a,\ b \in \mathcal{A}$ and cf(\A) denotes the set of class functionals on \A.
\end{definition}

\begin{lemma}Suppose (\A,\p) is a symmetric Frobenius algebra then 
\[
 \phi^{-1}(\mbox{cf}(\mathcal{A}))=Z(\mathcal{A})
 \ (= \mbox{centre of}\ \mathcal{A}).
 \]
\end{lemma}
\begin{proof}
As (\A,\p) is a symmetric Frobenius algebra, \p \ is a (\A ,\A) bimodule isomorphism between \A \ and \A$^{*}$.
We have
\begin{gather*}
 \gamma \in cf(\mathcal{A}) \\
 \Leftrightarrow \gamma(xy)=\gamma(yx),\ \forall\ x,y\in \mathcal{A}\\
 \Leftrightarrow (\gamma\cdot x)(y)=(x\cdot\gamma)(y),\ \forall \ x,y\in \mathcal{A}\\
 \Leftrightarrow\gamma\cdot x=x\cdot\gamma, \ \forall\  x \in \mathcal{A}\\
 \Leftrightarrow \phi^{-1}(\gamma \cdot x)=\phi^{-1}(x\cdot\gamma),\ \forall\ x\in \mathcal{A}\\
 \Leftrightarrow \phi^{-1}(\gamma)\cdot x=x\cdot\phi^{-1}(\gamma),\ \forall\ x\in \mathcal{A}\\
 \Leftrightarrow \phi^{-1}(\gamma)\in Z(\mathcal{A}).
\end{gather*}
Hence,  $\phi^{-1}\left( cf(\mathcal{A})\right)= Z(\mathcal{A})$.
This proves the lemma.
\end{proof}

\subsection{\textbf{Orthogonality relations for a Frobenius algebra}}

Let (\A,\p) be a Frobenius algebra, $M$ a simple left \A -module, and let
 \A$_{M}=\{a\in\mathcal{A}:am=0,\ \forall \ m \in M \}$.
 Then \A$_{M}$ is two sided ideal of \A, the $\mathbf{annihilator}$ of $M$ and
 \A$_{M}$ is a maximal two-sided ideal of \A.\
If $L$ is a left \A-module, Soc$_{M}(L)$ will denote the $M$-$socle$ of $L$, $i.e.$, the sum of all submodules of $L$ isomorphic to $M$.
We say that $L$ is $M$-$\textbf{homogeneous}$ if Soc$_{M}(L)=L$.
Also note that  $L$ is $M$-homogeneous if and only if $\mathcal{A}_{M}L=0$.

Let $\gamma\in \mathcal{A}^{*}$ such that 
$\gamma$(\A$_{M})$=0 then we say that 
 $\gamma \ \mathbf{belongs} \ to \  M$.
If $\gamma$\ belongs to $M$ then $a.\gamma$ and $\gamma.a$ also belong to $M$, $\forall\ a\in$ \A.

We say that $\chi\in \mathcal{A}^{*}$ \ is an \A -$\textbf{character}$ if $\chi$ \ is a character of a finite dimensional left \A -module.

Now onwards in this section \A \ denotes a Frobenius algebra (\A,\p).

\begin{lemma}
Let $M$ be a simple \A -module, and assume $\gamma$  $\emph{belongs}
$ to $M$. Set $\gamma^{*}=\phi^{-1}(\gamma)$. Then 
$\mathcal{A}_{M}\gamma^{*}=0$. In particular,
 if $N$ is any left \A -module,
 then \A $\gamma^{*}N$ is $M$-homogeneous.
\end{lemma}
\begin{proof}
 If $b \in\mathcal{A}_{M}$  then $b\gamma = 0$ ( $\because$ $(b\gamma)(\mathcal{A})=\gamma(\mathcal{A} b)\subseteq\gamma(\mathcal{A}_{M})=0$).
Therefore $b\gamma^{*} = \phi^{-1}(b\gamma)=0,\, b\in\mathcal{A}_{M}$, and so $\mathcal{A}_{M} \gamma^{*}$ = 0. 
For $N$  a left $\mathcal{A}$ -module,
$\mathcal{A}_{M}(\mathcal{A}\gamma^{*}N)=\mathcal{A}_{M}\gamma^{*}N=0$, so $\mathcal{A}\gamma^{*}N$ is $M$-homogeneous.
\end{proof}
 
\begin{theorem}{(Orthogonality relations).}
Let $M$ and $N$ be simple left \A -modules, and assume $\gamma$, $\mu$  belong to $M, N$  respectively, where  $\gamma$, $\mu \in\mathcal{A}^{*}$. 
Set $\gamma^{*}=\phi ^{-1}(\gamma)$.
If $\gamma^{*}N\neq 0$, \ then $M\cong N$.
In particular, if $\mu(\gamma^{*})\neq0$, \ then $M\cong N$.
 \end{theorem}
\begin{proof}
  
Suppose $\gamma^{*}N\neq0$. 
Since $N$\ is simple, 
$\mathcal{A}\gamma^{*}N =N$. By (1.5), $N$ must be $M$-homogeneous, $i.e.,M\cong N$.

In particular if $\mu(\gamma^{*})\neq0$, then $\gamma^{*}\notin
\mathcal{A}_{N}.\ i.e.,\, \gamma^{*}N\neq0$, so by above  $M\cong N$.
\end{proof}

\begin{corollary}\ Let $M,N,\gamma, \mu$ be as in (1.6), and set 
$$\gamma^{*}=\phi^{-1}(\gamma),\ \mu^{*}=\phi^{-1}(\mu).$$
If $\gamma^{*}\mu^{*}\neq0$,\ then $M\cong N.$
\end{corollary}
\begin{proof} Assume $\gamma^{*}\mu^{*}\neq0$.
Then $\gamma^{*}\mu\neq0$, and therefore for some $a\in$\A, 

0 $\neq(\gamma^{*}\mu)(a)=\mu(a\gamma^{*})=(\mu a)(\gamma^{*})$.
But $\mu a$\ belongs to $N$, so by (1.6), 
$M\cong N.$
\end{proof}

\begin{corollary}{(Orthogonality relations for characters)}\ Let, M, N be simple left \A-modules with characters $\chi$  and $\zeta$ respectively.
Set $\chi ^{*}$=\p$^{-1}$($\chi$).
If $\zeta$($\chi ^{*}$)$\neq$0, then M $\cong$ N and $\chi = \zeta$. 
In particular if $(a_{i})$, $(b_{i})$ are \p -dual bases for \A, and if $\sum_{i} \chi(a_{i})\zeta(b_{i})\neq$ 0, then M $\cong$ N and $\chi = \zeta$.
\end{corollary} 
\begin{proof} Corollary follows from lemma 1.2 and above corollary.
\end{proof}

\section{\textbf{ Frobenius $\star$-algebra}}
Now we will define a Frobenius $\star$-algebra and prove some results about it.
\begin{definition}
A Frobenius $\star$-algebra \A\ is a finite dimensional, associative, unital $\C$-algebra equipped with a  class functional $\lambda:\mathcal{A\rightarrow\C}$ and a map $\star:\mathcal{A}\rightarrow\mathcal{A}$, such that the following  holds:

$\bullet\ (a+b)^{\star} = a^{\star}+b^{\star}\ , \forall\ a, b \ \in \mathcal{A}.$

$\bullet\ (a^{\star})^{\star} = a,\ \forall\ a\ \in \mathcal{A}.$

$
\bullet\ (ab)^{\star}=b^{\star}a^{\star} \ , \forall \ a,b\in\mathcal{A}.$ 

$\bullet\ (\alpha a)^{\star}=\bar{\alpha} a^{\star} \,
     ,\forall \ a \in\mathcal{A} 
     ,\forall\ \alpha\in\mathbb{C}.
$

$\bullet$ $\langle a,b \rangle=\lambda(ab^{*})$ is a positive definite hermitian form on \A.

\end{definition}
\begin{remark}
As $\langle\cdot,\cdot\rangle$ is nondegenerate, so \la\ is a nondegenerate class functional.
Therefore \A\ is a symmetric Frobenius algebra.
\end{remark}
\begin{example}
$\mathcal{A} = \C[G]$ with \la\ is defined as
\la(g)=
$\begin{cases}
1\ \mbox{if}\ g=e ;\\
0\ \mbox{otherwise}.
\end{cases}$

$\star$-map is defined by
$g^{\star}=g^{-1},\ \forall\ g \in G$ and extended to \A \ by conjugate linearity.
\end{example}

Now onwards in this section \A\ denotes a Frobenius $\star$-algebra (\A,\la,$\star$).

\newpage
\begin{theorem}
 \A \ is a semi-simple algebra.
\end{theorem}
\begin{proof}
  By Artin-Wedderburn theorem it is enough to prove that \A \ is a semi-simple as a right \A -module.
To prove the theorem it is enough to show that for every right  ideal $\mathcal{I}$ of \A \ there exists another right ideal $\mathcal{J}$ of \A \ such that $\mathcal{A} = \mathcal{I}\oplus\mathcal{J}$, as right \A -modules.

Now for a fix right ideal $\mathcal{I}$ of \A.\ Define $\mathcal{J}$ as,
 $$\mathcal{J}=\{a\in\mathcal{A}:\langle a,y\rangle=0, \forall\ y\in\mathcal{I}\}$$
 $i.e.\ \mathcal{J}$ is the orthogonal compliment of $\mathcal{I}$ in \A. 
 As $\langle\, , \rangle $ is a positive definite hermitian form on
 \A, $\mathcal{A}=\mathcal{I}\oplus \mathcal{J}$ as a vector space.
To prove the theorem it is enough to show that $\mathcal{J}$ is a right ideal of \A. Let $r\in\mathcal{J}$ be any element and $a\in$ \A \ be any element and y $\in \mathcal{I}$. Consider 
$\langle ra,y\rangle=\lambda(ray^{*})=\langle r,ya^{*}\rangle=0$ 
( since $\mathcal{I}$ is right ideal of A)
$\Rightarrow \langle ra,y\rangle = 0,\
\forall\ y\in \mathcal{I}$
$\Rightarrow ra \in \mathcal{J},\  \forall\ a\in \mathcal{A},\ \forall\ r\in \mathcal{J}$.
This proves that $\mathcal{J}$ is a right ideal of \A, and the theorem follows.
\end{proof}

\noindent
$\bullet$Let Sim$(\mathcal{A})=$ set of all simple left \A-modules upto isomorphism.

By Theorem 2.4, \A \ is a semi-simple algebra. Therefore \A\ = $\oplus_{M \in Sim(\mathcal{A})} E(M)$, 
 $E(M) = M$-primary component of \A\ = Sum of all left ideals of \A \ isomorphic to $M$ as \A -module.

As \A\ is a semi-simple algebra, there exist primitive central orthogonal idempotents $e_{M}\in E(M)$, for every $M \in $ Sim(\A), such that
$\{e_{M}: M \in \mbox{Sim}(\mathcal{A})\}$ is a basis of $Z$(\A) and
$E(M)$ = \A$e_{M}$.

Let $M, N$ be simple \A -modules, $\chi_{M}, \chi_{N}$ 
characters of $M, N$  respectively.
 Set $\alpha _{M}=\phi^{-1}(\chi_{M}) \ \mbox{and} \ \alpha_{N}=\phi^{-1}(\chi_{N})$. As $\chi_{M}, \chi_{N} \in $ \ $cf$(\A), by  lemma (1.4) 
 $ \alpha_{M}, \alpha_{N} \in$ Z(\A). 
 Therefore, $\alpha_{M}=\sum_{L\in Sim(A)} c_{L}e_{L}$, and $\alpha_{N}=\sum_{L\in Sim(A)} d_{L}e_{L}$    where $c_{L},\ d_{L}\in \C$.

\begin{lemma} 
\ $e_{M}^{\star}=e_{M}$,\ $\forall \ M \in Sim(\mathcal{A})$.
\end{lemma}
\begin{proof} $e_{M}^{\star}e_{M}^{\star}=(e_{M}e_{M})^{\star}=e_{M}^{\star}$, and
for any $a\in $\A, $a.e_{M}^{\star}=(e_{M}a^{\star})^{\star}=(a^{\star}e_{M})^{\star}=e_{M}^{\star}.a$  and $e_{M}^{\star}e_{N}^{\star}=(e_{N}e_{M})^{\star}=0$\ if $M\neq N$.
 Therefore, $\{e_{M}^{\star}:M \in Sim(\mathcal{A})\}$ is a set of central orthogonal idempotents.
 Also, $1=\sum_{M\in Sim(\mathcal{A})}  e_{M}$
and $1^{\star}=1$. Therefore 
$$1=\sum_{M\in Sim(\mathcal{A})}  e_{M}=\sum_{M\in Sim(\mathcal{A})}  e_{M}^{\star}$$
$\Rightarrow$ $e_{M}^{\star}=e_{N}$ for some $N\in$ Sim(\A). But we know that,

$\langle e_{M},e_{M}\rangle\neq0$  $\Rightarrow$ \la$(e_{M}e_{M}^{\star})\neq0$ $\Rightarrow  e_{M}e_{M}^{\star}\neq0$ 
$\Rightarrow \ e_{M}=e_{M}^{\star}.$
\end{proof}

\begin{lemma} $\alpha_{M}=c_{M}e_{M}$.
\end{lemma}
\begin{proof} 
Let $a \in \mathcal{A}$ be any element then we can write $a = \sum_{L\in \mbox{Sim}(A)}a_{L}$.

Consider \p($e_{M}\alpha_{M}) = e_{M}\phi(\alpha_{M}) = e_{M}\chi_{M}$ then
$e_{M}\chi_{M}(a) = \chi_{M}(a e_{M}) = \chi_{M}(a_{M}) = \chi_{M}(a)$ (since $a_{N}\in \mathcal{A}_{M}$ if $M\neq N$).
Hence $e_{M}\chi_{M} = \chi_{M}$ $\Rightarrow$ $\phi^{-1}(e_{M}\chi_{M}) = \phi^{-1}(\chi_{M})$
$\Rightarrow$ $e_{M}\alpha_{M} = \alpha_{M}$.

Now $\alpha_{M} = e_{M}\alpha_{M}=\sum_{L\in \mbox{Sim}(\mathcal{A})}e_{M} c_{L}e_{L} = c_{M}e_{M}$.
Therefore $\alpha_{M} = c_{M}e_{M}$.
\end{proof}

\begin{lemma}
$\alpha_{M}.\alpha_{N}^{\star}=0$\ if $M\ncong N$.
\end{lemma}
\begin{proof}
$\alpha_{M}.\alpha_{N}^{\star}=c_{M}e_{M}\overline{c_{N}}e_{N}=c_{M}\overline{c_{N}}e_{M}e_{N}= 0$.
\end{proof}

\begin{theorem}{(Orthogonality Relations)}$\label{orthoforstar}$
Suppose $M$, $N$\ are simple \A -modules, 
$\chi_{M}$ and $\chi_{N}$\ denotes the characters of $M, N$ respectively.
Set $\phi^{-1}(\chi_{M})=\alpha_{M}\ and\ \phi^{-1}(\chi_{N})=\alpha_{N}$\ then 
\begin{gather*}
\langle \alpha_{M},\alpha_{N}\rangle=0,\, \mbox{if}\ M\ncong N \\
\langle\alpha_{M},\alpha_{M}\rangle=\abs{c_{M}}^{2} \langle e_{M},e_{M}\rangle.
\end{gather*}
Moreover, the irreducible characters of \A\ forms an orthogonal basis of the space of class functionals on \A.
\end{theorem}
\begin{proof}
  Follows from lemma (2.6) and lemma (2.7).
 \end{proof}
 
 \section{\textbf{Group graded Algebra}} 
Before proceeding further we fix some notations:\\
$\bullet$ \G \ denotes a finite group. \\
$\bullet$ $1_{G}$ denotes the identity element of group \G.\\
$\bullet$ Rings are assumed to contain unity.\\
$\bullet$ All modules in this section  are right modules.

All results and definitions in this section are found in \cite{Cliffordtheory}. 
\begin{definition}
A $G$-graded ring \A \ is a ring with an internal direct sum decomposition:
$
\mathcal{A} = \sum_{\sigma \in G} \mathcal{A} _{\sigma} \ (\mbox{as additive groups}),
$
where the additive subgroups  \A$_{\sigma}$ of  \A \ satisfy:
$\mathcal{A}_{\sigma}\mathcal{A}_{\tau}\subseteq \mathcal{A}_{\sigma\tau}, \ \forall\ \sigma,\tau \in  G.
$
\end{definition}
The above decomposition is called as a $G$-grading of \A , and its summands \A$_{\sigma}$ are called as $\sigma$-components of \A. \
The compoment of \A \ corresponding to $1_{G}$ is a subring of \A\ and
 \A \ is an unital module over \A$_{1_G}$. 
 Also \A$_{\sigma}$ is an unital module over \A$_{1_G}$ for $\sigma$ $ \in G$.
 If \A \ is a $G$-graded ring and $H$ is a subgroup of $G$\ then naturally we get a $H$-graded ring from \A, \A$_{H}=\sum_{h\in H} \mathcal{A}_{h}.$ 

\begin{example} 

Let $G$ be a finite group and $H$ be a normal subgroup of $G$. Then $\C[G]=\oplus_{\sigma H\in G/H }\C[\sigma H]$  is a $G/H$-graded ring.
\end{example}

\begin{definition}
 A $G$-graded module over a $G$-graded ring \A \ is an \A -module $M$ with an internal 
direct sum decomposition: $M =\sum_{g\in G} M_{g},$ (as $ \mathcal{A}_{1_{G}}$-modules),
where \A$_{1_{G}}$-modules $M_{g}$\ satisfy:
$M_{g}\mathcal{A}_{\sigma}\subseteq M_{g\sigma},\ \forall
 \ g,\sigma\in G$.
\end{definition}
Suppose $M$\ is a $G$-graded module over a $G$-graded ring \A\ then we get a natural $H$-graded \A$_{H}$-module $M_{H}=\sum_{h\in H}M_{h}$. 
A $G$-graded \A -submodule $N$ of $M$ is an \A -submodule which has $G$-grading with $\sigma$-component,
$N_{\sigma}=M_{\sigma}\cap N,\ \ \forall\ \sigma\in G$.

\noindent
$\bullet$ Let $M$,$N$ be two G-graded \A -modules then a G-graded homomorphism $\varphi:M\rightarrow N$ is a right \A -module homomorphism such that $\varphi({M_{\sigma}})\subseteq N_{\sigma}, \forall\ \sigma\in G.$
\subsection{Tensor Products:} We continue with  same notation as in (3.3).
For any \A$_{H}$-module $N$, the tensor product $N\otimes\mathcal{A}=
N \otimes _{\mathcal{A}_{H}} \mathcal{A}$ 
is naturally an \A -module, with multiplication determined by: $(n\otimes r)s=n\otimes (rs), \forall \ n\in N$ and $r$,$s$ $\in$ \A.
\begin{theorem}[see \cite{Cliffordtheory} Theorem 4.3]
If $N$\ is a $H$-graded \A$_{H}$-module, then there is exactly one $G$-grading for the \A -module $N\otimes\mathcal{A}$\  making the latter into a 
$G$-graded \A  -module such that $\cdot \otimes 1 $\ is an isomorphism:
$$\cdot \otimes 1:N\xrightarrow{\sim}(N\otimes\mathcal{A})_{H}\ 
(\mbox{as}\ \mbox{H-graded}\ \mathcal{A}_{H} \mbox{-modules}).$$
The $\sigma$-components of this $G$-grading are given by:
$$(N\otimes\mathcal{A})_{\sigma}=\sum_{(h,g)\in P(\sigma)}(N_{h}\otimes 1)\mathcal{A}_{g}, \ \forall \ \sigma\in G$$
where $P(\sigma)=\{ (h,g)\in H\times G : hg=\sigma\}.$
\end{theorem}

\subsection{Induction}
The tensor product $N\otimes\mathcal{A}$ is the most natural $G$-graded \A -module associated with a $H$-graded \A$_{H}$-module.
However it is not the one that we use for induction.
The module that we use for induction is a quotient of $N\otimes\mathcal{A}$ by a G-graded \A -submodule of it. 

Let $M$ be a $G$-graded \A -module.
We say that an $G$-graded \A -submodule $U$ of $M$ is $\textbf{H-null}$ if $U_{H}=0$. 
Define a $H$-null $\textbf{Socle},\ S_{H}(M)$  of $M$ as the largest(under inclusion) $H$-null $G$-graded \A-submodule of $M$.

Now returning to our $H$-graded \A$_{H}$-module $N$ and associated $G$-graded \A -module $N\otimes\mathcal{A}$, we define $\textit{induced module}\  N\bar{\otimes}\mathcal{A}$ to be the $G$-graded \A -quotient module:
$$ N\overline{\otimes}\mathcal{A}=(N\otimes\mathcal{A})/S_{H}(N\otimes\mathcal{A}).$$
\subsection{Action of G on simple \A$_{1_G}$-modules}
\begin{definition}
A simple $G$-graded \A -module $M$ is a $G$-graded \A -module with no proper $G$-graded submodule. $\textit{i.e.}$, the only $G$-graded \A -submodules of $M$ are 0 and itself.
\end{definition}

Let $\Sigma$ = the set of all $G$-graded \A -modules. We define an action of $G$ on $\Sigma$ as,
for $M\in\Sigma$, $\tau \in G$, 
\begin{gather*}
M^{\tau}=M\ (\mbox{as}\ \mathcal{A}\mbox{-modules})\\
(M^{\tau})_{\sigma}=M_{\tau \sigma}, \ \forall \ \sigma\in G. 
\end{gather*}

\noindent
$\bullet$ If $M$ is a simple $G$-graded module then $M^{\tau}$ is also simple a $G$-graded module for all $\tau\in G$.

Define the $\textit{support}$ of a simple \G -graded \A -module $M$ as,
$$ Supp(M)=\{ \sigma\in G: M_{\sigma}\neq0\}.$$
\begin{lemma}[see \cite{Cliffordtheory} Lemma 6.1]
If $M$ is a simple $G$-graded \A -module, then $M_{H}$ is either 0 or a simple $H$-graded \A$_{H}$-module. In the latter case $M=M_{H}\mathcal{A}$ is H-generated and there is unique isomorphism
$$\delta=\delta_{m}:M\xrightarrow{\sim}M_{H}
\overline{\otimes}\mathcal{A}, \  (\mbox{as}\ G \mbox{-graded}\ \mathcal{A} \mbox{-modules})$$
such that:
$$\delta(mr)=m\bar{\otimes}r,\ \forall \ m\in M_{H}, \and r\in \mathcal{A}.$$
\end{lemma}

\begin{corollary}
Let $M$ be a simple $G$-graded \A -module such that $M_{\tau}\neq 0$ for some $\tau\in G$. Then $M_{\tau}$ is a simple $A_{1_{G}}$-module.
\end{corollary}

\begin{lemma}[see \cite{Cliffordtheory} Lemma 6.3]
If $N$ is a simple $H$-graded \A$_{H}$-module, then any $G$-graded \A -submodule $U$ of $N\otimes\mathcal{A}$ is either contained in $S_{H}(N\otimes\mathcal{A})$ or equal to $N\otimes\mathcal{A}$. Hence $N
\overline{\otimes}\mathcal{A}$ is  a simple $G$-graded \A -module with $(N\overline{\otimes}\mathcal{A})_{H}\neq 0$.
\end{lemma}

We define a partial action of \G \ on simple \A $_{1_{G}}$-modules as follows:\\
For $\sigma\in G$ and for simple \A $_{1_{G}}$-module $V$,
$$V^{\sigma}=((V\overline{\otimes}\mathcal{A})^{\sigma})_{1_{G}} =(V\bar{\otimes}\mathcal{A})_{\sigma}$$
by lemma (3.8) $V^{\sigma}$ is a simple $\mathcal{A}_{1_{G}}$-module, if $\sigma \in Supp(V\overline{\otimes}\mathcal{A})$ and $V^{\sigma} = 0$, if $\sigma \notin$ Supp($V\overline{\otimes}\mathcal{A}$).

We say that two simple \A$_{1_{G}}$-modules $U$ and $V$ are \G -$\textit{conjugate}$ if there exists a $\tau\in$ \G \ such that $U^{\tau}$ is \A$_{1_{G}}$- isomorphic to $V$.

\noindent
$\bullet\ G\{U\}=\{\sigma\in G:U^{\sigma}\cong U \ \mbox{as}\ \mathcal{A}_{1_{G}} \mbox{-modules}\}.$
\begin{lemma}[see \cite{Cliffordtheory} Theorem 7.17]
Let $U$ be a simple \A $_{1_{G}}$-module such that for all $\tau\in G$, $U$ and $U^{\tau}$ are isomorphic as \A$_{1_{G}}$-modules then $U\overline{\otimes}\mathcal{A}$ is $U$-primary.
\end{lemma}

Let $V$ be a simple $A_{1_{G}}$-module then we say that an \A -module $\textbf{M lies over}$ $\textbf{ V}$  if it satisfies: \\
$\bullet$ $M$ is a semisimple as an \A$_{1_{G}}$-module. \\
$\bullet$ The $U$-primary component $M\{ U\}$ of $M$ is 0 for any simple \A$_{1_{G}}$-module $U$ not conjugate to $V$.\\
$\bullet$ The \A$_{1_{G}}$-submodule $M\{ U\}\mathcal{A}_{\tau}$ of $M$ is 0 whenever $\tau\in G$ and $U^{\tau}=0$.

\begin{definition}{(Graded endomorphism ring)}
Let $U$ be a simple \A$_{1_{G}}$-module, a graded endomorphism ring $\mathcal{E}(U)$ corresponding to $U$ is defined as $\mathcal{E}(U)=\sum_{\sigma\in G}\mathcal{E}(U)_{\sigma}$, where
$$\mathcal{E}(U)_{\sigma}=\{\phi\in End_{\mathcal{A}}(U\overline{\otimes}\mathcal{A}): \phi\left((U\overline{\otimes}\mathcal{A})_{\tau}\right)\subseteq(U\overline{\otimes}\mathcal{A})_{\sigma\tau}, \forall\ \tau \in G\}$$
\end{definition}

\noindent
$\bullet\ \mathcal{F}(U)$ denotes a End$_{\mathcal{A}_{1_{G}}}(U)$ for any \A$_{1_{G}}$-module $U$. 
 
\noindent
$\bullet\ \mathcal{F}(U)\cong \mathcal{E}(U)_{1_{G}}$ as a ring.
 \begin{theorem}[see \cite{Cliffordtheory} proposition 9.4]
 For any $\sigma\in G$ the non zero elements \p\ of $\mathcal{E}(U)_{\sigma}$ are precisely the isomorphisms:
 
 $$\phi:U\overline{\otimes}\mathcal{A}\xrightarrow{\sim}(U\overline{\otimes}\mathcal{A})^{\sigma}\ (\mbox{as graded }\ \mathcal{A}\mbox{-modules})$$
 Hence such a \p \ exists if and only if $\sigma\in G\{U\}$.
 Any such \p \ lies in $U(\mathcal{E}(U))\cap\mathcal{E}(U)_{\sigma}$.
 Therefore, $\mathcal{E}(U)$ is equal to its $G\{U\}$-graded subring $\mathcal{E}(U)_{G\{U\}}$, which is a crossed product of $G\{U\}$ over the division ring $\mathcal{E}(U)_{1_{G}}\cong End_{\mathcal{A}_{1_{G}}}(U).$
 \end{theorem}

\begin{theorem}[see \cite{Cliffordtheory} Theorem 10.5]
If $M$ is any $\mathcal{E}(U)$-module, then there is an \A$_{1_{G}}$-homomorphism
 $\lambda_{M}=M\otimes(\cdot\ \overline{\otimes}\ 1):M\otimes_{\mathcal{F}(U)}U\rightarrow M\otimes_{\mathcal{E}(U)}(U\overline{\otimes}\mathcal{A})$ 
 sending $m\otimes u$ to $m\otimes(u\bar{\otimes}1)$, for any $m\in M$ 
 and $u\in U$.
 This $\lambda_{M}$ is an \A$_{1_{G}}$-isomorphism of $M\otimes_{\mathcal{F}(U)} U$ onto the $U$-primary component $M\otimes_{\mathcal{E}(U)}(U\overline{\otimes}\mathcal{A})\{U\}$ of $M\otimes_{\mathcal{E}(U)}(U\overline{\otimes}\mathcal{A}).$
\end{theorem}

The category $\textbf{Mod}(\mathcal{A}\mid U)$ of \A -modules which lie over $U$ and the category $\textbf{Mod}(\mathcal{E}(U))$ of $\mathcal{E}(U)$-modules are both abelian categories.
Let $\cdot \langle 
U\rangle$ denote the additive functor $\mbox{Hom}_{\mathcal{A}}(U\overline{\otimes}\mathcal{A}, \cdot)$ 
 from $\textbf{Mod}(\mathcal{A}|U)$ to $\textbf{Mod}(\mathcal{E}(U))$ and $\left( \cdot \right)^{\mathcal{A}}=\cdot\ \otimes_{\mathcal{E}(U)}(U\overline{\otimes}\mathcal{A})$ be the additive functor in the opposite direction, \textit{i.e.}, from 
 $\textbf{Mod}(\mathcal{E}(U))$ to  $\textbf{Mod}(\mathcal{A}|U)$, 
\begin{theorem}[see \cite{Cliffordtheory} Theorem 10.5]
The additive functors 
$$
\cdot \langle U\rangle= \mbox{Hom}_{\mathcal{A}}(U\overline{\otimes}\mathcal{A}, \cdot), \ \mbox{and}\ \left( \cdot \right)^{\mathcal{A}}=\cdot\ \otimes_{\mathcal{E}(U)}(U\overline{\otimes}\mathcal{A}) 
$$
form an equivalence between the abelian categories $\textbf{Mod}(\mathcal{A}|U)$ and $\textbf{Mod}(\mathcal{E}(U))$ for any finite group \G, any \G -graded ring \A \ and any simple \A$_{1_{G}}$ simple module $U$.
\end{theorem}

\begin{theorem}[see \cite{Cliffordtheory} Theorem 12.10] 
Let $M$ be a simple \A -module. 
If the group $G$ is  finite,  then there is some \A -monomorphism of $M$ into 
simple $G$-graded \A -module $N$. 
Hence $M$ lies over a some simple $\mathcal{A}_{1_{G}}$-module $U$.
\end{theorem}

\section{\textbf{Twisted orthogonality relations for a $\Z/m\Z$-graded Frobenius $\star$-algebra} }

In this section we consider a $\Z/m\Z$-graded Frobenius $\star$-algebra (\A,\la,$\star$).
\begin{definition}
 A $\Z/m\Z$-graded Frobebius $\star$-algebra \A\ is  a $\Z/m\Z$-graded finite dimensional associative unital $\C$-algebra equipped with a map $\star : \mathcal{A}\rightarrow\mathcal{A}$ and a class functional $\lambda_{0} : \mathcal{A}_{0}\rightarrow \mathcal{C}$ which is extended linearly to \A\ by zero outside of $\mathcal{A}_{0}$ we denote this extension by \la\  , such that the following are holds:
 
 $\bullet\ (a+b)^{\star} = a^{\star}+b^{\star}\ , \forall\ a, b \ \in \mathcal{A}.$

$\bullet\ (a^{\star})^{\star} = a,\ \forall\ a\ \in \mathcal{A}.$

 $
\bullet\ (ab)^{\star}=b^{\star}a^{\star} \ , \forall \ a,b\in\mathcal{A}.$ 

 $\bullet\ (\alpha a)^{\star}=\bar{\alpha} a^{\star} \,
     ,\forall \ a \in\mathcal{A} 
     ,\forall\ \alpha\in\mathbb{C}.
$

$\bullet(\mathcal{A}_{r})^{\star}\subseteq \mathcal{A}_{-r}\, \ \forall\ r \ \in\ \Z/m\Z.$

$\bullet$  $\langle a,b\rangle= \lambda(ab^{\star})$ is a positive definite hermitian form on \A.
\end{definition}

\begin{remark}
If \A\ is a $\Z/m\Z$-graded Frobenius $\star$-algebra then (\A$_{0}$, \la$_{0}, \star)$ is  a Frobenius $\star$-algebra and  (\A, \la, $\star$) is also a Frobenius $\star$-algebra.

\end{remark}

\begin{lemma}
Let \A = $\oplus_{r\in \Z/m\Z}\mathcal{A}_{r}$ be a $\Z/m\Z$-graded Frobenius $\star$-algebra as above.

\noindent
$(i)$ Let $M = \oplus_{r\in \Z/m\Z} M_{r}$, with each $M_{r} \subseteq \mathcal{A}_{r}$ be a $\Z/m\Z$-graded left \A -submodule of  \A. Suppose that $M_{0} = 0$, 
i.e. $M$ is null in the sense of \cite{Cliffordtheory} section 5. Then $M_{r} = 0$ for all $r\in \Z/m\Z$. In other words, the null socle of the $\Z/m\Z$-graded left \A -module \A\ is zero.

\noindent
$(ii)$ Let $M$ be any (left) \A $_{0}$-module. Then the null socle of the $\Z/m\Z$-graded \A -module \A$ \otimes_{\mathcal{A}_{0}} M$ is zero. In other words, the \A -module induced by the \A $_{0}$-module $M$ in the sense  section 3.2  is equal to \A $\otimes _{\mathcal{A}_{0}} M$.
\end{lemma}
 \begin{proof}
 Let $r \in\Z/m\Z$. Since $M\subseteq\mathcal{A}$ is a graded submodule, we  
 have $m \in M_{r}$ then $m^{\star}\in \mathcal{A}_{-r}$, $m^{\star}m\in M_{0}=0 $, therefore $\langle m, m^{\star} \rangle =  \lambda_{0}(mm^{\star}) = 0$. 
 Hence we must have $M_{r}=0$, (since $\langle\cdot,\cdot\rangle$ is positive definite) this proves part $(i)$. 
 
 As $\mathcal{A}_{0}$ is a semisimple algebra it is enough to prove part $(ii)$ for $M\in \mbox{Sim}(\mathcal{A}_{0})$.
 Then we know that $M$ occurs as a direct summand in regular representation of $\mathcal{A}_{0}$. Then the statement follows from $(i)$.

\end{proof}  
 
\begin{definition} {(Centralizer of $\mathcal{A}_{0}$ in $ \mathcal{A}$)} The centralizer of $\mathcal{A}_{0}$ in \A\ is denoted by $Z_{\mathcal{A}_{0}}(\mathcal{A})$ and defined as $Z_{\mathcal{A}_{0}}(\mathcal{A}) = \{a \in \mathcal{A} : ab=ba  \ \forall\  b \in \mathcal{A}_{0}\}$. 
 \end{definition}
 \begin{remark}
 $Z_{\mathcal{A}_{0}}(\mathcal{A}) = \oplus_{r \in \Z/m\Z} Z_{\mathcal{A}_{0}}(\mathcal{A}_{r})$
 
 \end{remark}

Now we will give a structure of simple $Z_{\mathcal{B}}(\mathcal{A})$-module for any semisimple algebra \A\ and its semisimple subalgebra $\mathcal{B}$. 
 
\begin{lemma}
Let $\mathcal{A}$ be a semisimple $C$-algebra and $\mathcal{B}$ be its subalgebra.
We have an isomorphism $Z_{\mathcal{B}}(\mathcal{A}) = \oplus_{
U,V} End(Hom_{B}(U, V ))$, where the sum
is taken over all pairs $(U, V ) \in$ Sim($\mathcal{B})$ $\times$ Sim(\A) satisfying Hom$_{B}(U, V )  \neq \{0\}$.
\end{lemma}
 In other words, the algebra $Z_{\mathcal{B}}(\mathcal{A})$ is semisimple and the irreducible $Z_{\mathcal{B}}(\mathcal{A})$-modules are
precisely the nonzero multiplicity spaces Hom$_{B}(U, V )$.
\begin{proof}
Since \A \ is semisimple, it can be identified with $\oplus_{V \in \mbox{Sim}(\mathcal{A})}\mbox{End}(V )$, it follows from structure theorem for a semisimple algebra.
 In this realization, we have $Z_{B}(\mathcal{A}) = \oplus_{V \in \mbox{Sim}(\mathcal{A})}\mbox{End}_{B}(V )$. 
 But we know that 
End$_{B}(V ) = \oplus_{U \in \mbox{Sim}(B)}\mbox{End}(Hom_{B}(U, V ))$,
where the summation is taken over all $U \in \mbox{Sim}(B)$ such that Hom$_{B}(U, V ) \neq {0}$.
Hence the proof.
\end{proof} 
 
 \begin{corollary}
 Every simple $Z_{\mathcal{A}_{0}}(\mathcal{A}_{0})$-module is isomorphic to Hom$_{\mathcal{A}_{0}}(U, U)$ for some $U \in $ Sim$(\mathcal{A}_{0})$.
 \end{corollary}

To illustrate the general phenomenon to be discussed in this section, let us consider the following example.

Let $G$\ be a finite group, $\sigma\in$ Aut(G) and order of $\sigma$ is $m$ then we get an action of $\Z/m\Z=\langle\sigma\rangle \ \mbox{on} \ G$, this action induces an action of $\Z/m\Z$\ on Irr($G$).
Which is given by,$$\mbox{for}\ \rho\in \mbox{Irr}(G), \ \rho^{\sigma}=\rho \circ \sigma$$ and so on.

Let $\rho\in \mbox{Irr}(G)^{\Z/m\Z}$ \ then $ \rho\cong\rho^{\tau},$ \ $\forall \ \tau\in \Z/m\Z$. 
Suppose $\phi$\ is an isomorphism between $\rho$ and $\rho^{\sigma}$ then,
\begin{gather*}
\rho \circ \phi=\phi \circ \rho^{\sigma}  
\Rightarrow \rho(g) \circ\phi =\phi \circ\rho^{\sigma}(g),\ \forall\ g\in G \\
\Rightarrow\rho(\sigma(g))\circ\phi =\phi \circ\rho^{\sigma}(\sigma(g)),\ \forall\ g\in G \
\Rightarrow\rho^{\sigma} \circ\phi =\phi \circ\rho^{\sigma^{2}} \\
\mbox{but}\ \rho^{\sigma}=\phi^{-1} \circ\rho \circ\phi
\Rightarrow \phi^{-1} \circ\rho \circ\phi \circ \phi=\phi \circ \rho^{\sigma^{2}}
\Rightarrow\rho \circ\phi^{2}=\phi^{2} \circ \rho^{\sigma^{2}}.
\end{gather*}
Therefore $\phi^{2}$ is an isomorphism between $\rho \ \mbox{and} \ \rho^{\sigma^{2}}$.
Similarly we can show that $\phi^{r}$\ is an isomorphism between $\rho$ and $\rho^{\sigma^{r}}$.
Therefore for $r=m$  we get that,

$\rho\circ \phi^{m}=\phi^{m}\circ\rho^{\sigma^{m}}=\phi^{m}\circ\rho\ (\mbox{since}\ o(\sigma)=m)$
\ $\Rightarrow \phi^{m}=cI, \ \mbox{ for\ some}\ c\in \C$ (by Schur's lemma).
Replace $\phi$ by $c^{1/m}\phi$, then we can assume that $\phi^{m}=I.$

Suppose $\tilde{G} = G\rtimes\Z/m\Z = G \rtimes \langle\sigma\rangle $ then we can extend $\rho\in $ Irr(\G)$^{\Z/m\Z}$ \ to a representation $\tilde{\rho}$\ of $\tilde{G}$.
This can be done in the following way, 
\begin{gather*}
\tilde{\rho}(g)=\rho(g),\ \mbox{for} \, g\in G \\   \tilde{\rho}(\sigma)=\phi,\ \mbox{for} \ \sigma\\  
 \tilde{\rho}(g,\sigma^{r})=\rho(g)\circ\phi^{r} ,\ \mbox{for}\ (g,\sigma^{r})\in \tilde{G}.
\end{gather*}
it is easy to show that $\tilde{\rho}\in\mbox{Irr}( \tilde{G}).$
This extension has $m$ choices since $\phi$ has $m$ choices and their characters are differ by a $m^{th}$\ root of unity.

Let $\rho_{1},  \rho_{2 }\in$ Irr(\G)$^{\Z/m\Z}$ \ then 
we get $\tilde{\rho_{1}},\tilde{\rho_{2}}\in\mbox{Irr}(\tilde{G})$,
denote $\chi_{\rho_{i}}$= character of
$\rho_{i}$, $\chi_{\tilde{\rho_{i}}}=\mbox{character\ of}\ \tilde{\rho_{i}}$.
There are well known orthogonality relations for the irreducible characters of a finite group: 
\begin{equation*}
\langle\chi_{\rho_{1}},\chi_{\rho_{2}}\rangle=\frac{1}{\abs{G}}\sum_{g \in G}\chi_{\rho_{1}(g)}\overline{\chi_{\rho_{2}}(g)}=
\begin{cases}
0 & \mbox{if}\ \rho_{1}\ncong\rho_{2};  
 \\
1 & \mbox{if}\ \rho_{1 }\cong \rho_{2}. 
\end{cases}
\end{equation*}
and
\begin{equation*}
\langle\chi_{\tilde{\rho_{1}}},\chi_{\tilde{\rho_{2}}}\rangle=\frac{1}{\abs{\tilde{G}}}\sum_{g \in G}\chi_{\tilde{\rho_{1}}(g)}\overline{\chi_{\tilde{\rho_{2}}}(g)}=
\begin{cases}
0 & \mbox{if}\ \tilde{\rho_{1}}\ncong\tilde{\rho_{2}};  
 \\
1 & \mbox{if}\ \tilde{ \rho_{1 }}\cong \tilde{\rho_{2}}. 
\end{cases}
\end{equation*}
We know that for an extension of $\rho\in$ Irr(\G)$^{\Z/m\Z}$ to $\tilde{G}$ has $m$-choices and their characters are differ by $m^{th}$ root of unity. Denote all extension of $\rho$ by 
$\tilde{\rho},\ \tilde{\rho_{\omega}},\ \tilde{\rho_{\omega^{2}}},\ \cdots,\ \tilde{\rho_{\omega^{m-1}}}$, where $\omega$
is a primitive $m^{th}$ root of unity.
Note that  $\tilde{\rho_{\omega^{r}}}$\ and $\tilde{\rho_{\omega^{k}}}$\ are not isomorphic, for 0 $\leq r < k\leq m-1$ (since their characters are distinct).

\begin{definition}
 Let $\rho\in \mbox{Irr}(G)^{\Z/m\Z}$, the twisted character $\tilde{\chi}_{\rho}$  corresponding to $\rho$ is defined as, 
 $$ \tilde{\chi}_{\rho}(g)\coloneqq\chi_{\tilde{\rho}}(g,\sigma).$$ 
\end{definition}

By using the orthogonality relations for \G\ and $\tilde{G}$ we get, for $\rho_{1},\rho_{2}\in$ Irr($G$)$^{\Z/m\Z}$

\begin{equation*}
\langle\tilde{\chi}_{\rho_{1}},\tilde{\chi}_{\rho_{2}}\rangle=
 \begin{cases} 1\ \ \mbox{if}\ \rho_{1}\cong\rho_{2}; \\
 0 \ \ \mbox{if}\ \rho_{1}\ncong\rho_{2}.
  \end{cases}
  \end{equation*}\\
The above relations are called as $\textbf{Twisted orthogonality relations for irred-}$\\ $\textbf{ucible characters of G}$.
The proof is similar to the  general case that we will see in the next part.

Now we put the above discussion in the language of graded algebras.
Let \A\ = $\mathbb{C} [\tilde{G}]$ then \A$_{0}=\mathbb{C} [G]$.
Define  $\star:\mathcal{A}\rightarrow\mathcal{A}$ as, $\tilde{g}^{\star} = \tilde{g}^{-1}$ and extending by conjugate linearity to $\mathcal{A}$ and define $\lambda_{0}: \mathcal{A_{\circ}} \rightarrow \mathbb{C}$ as, $\lambda_{0}(e)=1 \ \mbox{and}
\ \lambda_{0}(g)=0\ \mbox{if}\ g \neq e$\ and extend linearly to \A$_{0}$. 
Now extend $\lambda_{0}$ to \A\ by zero outside of \A$_{0}$.
Then the following form on \A\ is a positive definite hermition form on \A.
$$\langle x,y\rangle=\lambda(xy^{\star})
$$
Therefore \A\ becomes a $\Z/m\Z$-graded Frobenius $\star$-algebra.
We know that irreducible representations of G and irreducible representations of $\mathbb{C}[G]\ $\ are same. 
Therefore if $M$ is a simple $\mathcal{A}_{o}$-module
and $M\in \mbox{Irr}(\mathcal{A}_{o})^{\Z/m\Z}$\ then we can make $M$ as a simple \A -module. 
Let $\chi_{M}$ denotes the character of $M$  as an $\mathcal{A}_{0}$-module, $\chi_{\tilde{M}}$ denotes the character of $M$ as an \A  -module  and the restriction of  $\chi_{\tilde{M}}$ to  $\mathcal{A}_{1}$ is denoted $\tilde{\chi}_{M}$. 
It follows from the above discussion that for $M,N$ \ simple \A -modules,
$$\langle\tilde{\chi}_{M},\tilde{\chi}_{N}\rangle=
\begin{cases}
1\ \ \mbox{if}\ M\cong N ;\\
0\ \ \mbox{if}\ M\ncong N.
\end{cases}
$$.

$\mathbf{General \ Case}$\ Suppose \A\ is a $\Z/m\Z$-graded Frobenius $\star$-algebra.

If $M$\ is a simple \A$_{o}$-module then $\mathcal{F}(M)=$ End$_{\mathcal{A}_{\circ}}(M)=\C$,
and if $M\in$ $\mbox{Sim}(\mathcal{A}_{o})^{\Z/m\Z}$ then 
$\mathcal{E}(M)=\oplus_{\tau\in\Z/m\Z}\C$.
By theorem (3.11), $\mathcal{E}(M)$ is a crossed product of $\Z/m\Z$ over $\C$\ (which is isomorphic to $\C[\Z/m\Z]$). 
Therefore upto isomorphism there are $m$ simple modules over $\mathcal{E}(M)$, all are 1-dimensional and their characters are differ by $m^{th}$ roots of unity. We list all these simple modules by $\C,\ \C_{\omega},\ \cdots,\ \C_{\omega^{m-1}}.$

By using results from section 3 we get $\C\otimes_{\C} M \cong\C\otimes_{\mathcal{F}(M)} M \cong \C\otimes_{\mathcal{E}(M)}(M\overline{\otimes}\mathcal{A})$ as \A$_{\circ}$-modules and $\C\otimes_{\mathcal{F}(M)} M \cong M$ as \A$_{\circ}$-modules.
In this way we make $M$ as a simple \A -module and the above isomorphism has $m$-choices.

We denote all possible extensions of  $M$ to an \A-module by
 $\tilde{M},\ \tilde{M}_{\omega},\ \cdots,\\ \tilde{M}_{\omega^{m-1}}$ and their characters as \A -modules by 
 $\chi_{\tilde{M}},\ \chi_{\tilde{M}_{\omega}},\ \cdots,\ \chi_{\tilde{M}_{\omega^{m-1}}}$.
 
 Set $\alpha_{\tilde{M}}=\phi ^{-1}(\chi_{\tilde{M}}),
 \alpha_{\tilde{M_{\omega}}}=\phi^{-1}(\chi_{\tilde{M_{\omega}}}),\ \cdots,\ \alpha_{\tilde{M_{\omega^{m-1}}}}=\phi^{-1}(\chi_{\tilde{M_{\omega^{m-1}}}}).$ Where $\phi$ is an isomorphism between \A\ and \A$^{*}$ as in the definition of a Frobenius algebra.

Note that $\chi_{\tilde{M_{\omega^{r}}}}(a_{s})=\omega^{sr}\chi_{M}(a_{s}),\ \forall\ a_{s}\in \mathcal{A}_{s},  0\leq r,s\leq m-1$
and $\tilde{M_{\omega^{s}}}\ \mbox{and}\ \tilde{M_{\omega^{r}}}$\ are not isomorphic, whenever $0\leq r < s \leq m-1$, since their characters are distinct.
By theorem $\ref{orthoforstar}$\ we have the following orthogonality relations for irreducible characters of \A\ and for irreducible character of $\mathcal{A}_{o}$.

Now fix $M,N\in \mbox{Sim}(\mathcal{A}_{o})^{\Z/m\Z}$ then by theorem $\ref{orthoforstar}$,
we get that 
\begin{gather*}
\langle\alpha_{M},\alpha_{N}\rangle =0,\ \ \mbox{if}\ M\ncong N  \\
\langle\alpha_{M},\alpha_{M}\rangle=\abs{c_{M}}^{2} \langle e_{M},e_{M}\rangle
\end{gather*}

and 
\begin{gather*}
\langle\alpha_{\tilde{M}},\alpha_{\tilde{N}}\rangle=0,\ \ \ \mbox{if}\ \tilde{M}\ncong\tilde{N} \\
\langle\alpha_{\tilde{M}},\alpha_{\tilde{M}}\rangle=\abs{c_{M}}^{2} \langle e_{\tilde{M}},e_{\tilde{M}}\rangle.
\end{gather*}
As $\alpha_{\tilde{M}}\in\mathcal{A}$\ we can write 
$\alpha_{\tilde{M}}=\sum_{r=0}^{m-1}\alpha_{M}^{(r)}$,\ where $\alpha_{M}^{(r)}\in\mathcal{A}_{r}$ and
$\alpha_{\tilde{M_{\omega^{s}}}}=\sum_{r=0}^{m-1}\alpha_{M_{\omega}}^{(r)} =\sum_{r=1}^{n}\omega^{sr}\alpha_{M}^{(r)}$.

Therefore,
\begin{gather*}
\langle\alpha_{\tilde{M}},\alpha_{\tilde{M}} \rangle=\sum_{r=0}^{m-1}\langle\alpha_{M}^{(r)},\alpha_{M}^{(r)}\rangle=\abs{c_{M}}^{2} \langle e_{M},e_{M}\rangle \\
\langle\alpha_{\tilde{M}},\alpha_{\tilde{M_{\omega}}}\rangle=\sum_{r=0}^{m-1}\langle\alpha_{M}^{(r)},\alpha_{M_{\omega}}^{(r)}\rangle=\sum_{r=0}^{m-1}\omega^{-r}\langle\alpha_{M}^{(r)},\alpha_{M}^{(r)}\rangle=0\\
\langle\alpha_{\tilde{M}},\alpha_{\tilde{M_{\omega^{2}}}}\rangle=\sum_{r=0}^{m-1}\langle \alpha_{M}^{(r)},\alpha_{ M_{\omega^{2}}}^{(r)}\rangle=\sum_{r=0}^{m-1}\omega^{-2r}\langle\alpha_{M}^{(r)},\alpha_{M}^{(r)}\rangle=0\\
\vdots \\
\langle\alpha_{\tilde{M}},\alpha_{\omega^{(m-1)}\tilde{M}}\rangle=\sum_{r=0}^{m-1}\langle\alpha_{M}^{(r)},\alpha_{\omega^{(m-1)} M}^{(r)}\rangle=\sum_{r=0}^{m-1}\omega^{-(m-1)r}\langle\alpha_{M}^{(r)},\alpha_{M}^{(r)}\rangle=0
\end{gather*}
Put $x_{r}=\langle\alpha_{M}^{(r)},\alpha_{M}^{(r)}\rangle$, then we get a system of equations which is expressed in a matrix form as,\\
$\underbrace{\begin{bmatrix}
1&1&\cdots&1\\
1&\omega^{-1} &\cdots&\omega^{-(m-1)}\\
\vdots\\
1&\omega^{-(m-1)}&\cdots&\omega^{-(m-1)^{2}}
\end{bmatrix} 
}_{\mathbf{P}}
\begin{bmatrix}
x_0\\
x_1\\
\vdots\\
x_{m-1}
\end{bmatrix} 
=
\begin{bmatrix}
\abs{c_{M}}^{2} \langle e_{\tilde{M}},e_{\tilde{M}}\rangle  \\
0\\
\vdots\\
0
\end{bmatrix} $
 
 The matrix $\mathbf{P}$\ is a Vandermonde matrix  with det($\mathbf{P})\neq0$ i.e. $\mathbf{P}$ is invertible  and it is easy to compute first column of $\mathbf{P^{-1}}$ \ which is,
 $\begin{bmatrix}
\dfrac{1}{m}  \\
\dfrac{1}{m}\\
\vdots\\
\dfrac{1}{m}
\end{bmatrix}$.

Therefore the unique solution to the above system is, $$x_{r}=\dfrac{\abs{c_{M}}^{2} \langle e_{\tilde{M}},e_{\tilde{M}}\rangle}{m},\ \mbox{ for }\ 0\leq r \leq m-1 $$ 
\begin{equation}
\langle\alpha_{M}^{(r)},\alpha_{M}^{(r)}\rangle=\dfrac{\abs{c_{M}}^{2} \langle e_{\tilde{M}},e_{\tilde{M}}\rangle}{m}, \ \mbox{ for } \ 0\leq r \leq m-1.
\end{equation}

 Also for $M\neq N\in$ Sim(\A$_{\circ})^{\Z/m\Z}$,
\begin{gather*}
\langle\alpha_{\tilde{M}},\alpha_{\tilde{N}} \rangle=\sum_{r=o}^{m-1}\langle\alpha_{M}^{(r)},\alpha_{N}^{(r)}\rangle=0 \\
\langle\alpha_{\tilde{M}},\alpha_{\tilde{N_{\omega}}}\rangle=\sum_{r=o}^{m-1}\langle\alpha_{M}^{(r)},\alpha_{ N_{\omega}}^{(r)}\rangle=\sum_{r=o}^{m-1}\omega^{-1}\langle\alpha_{M}^{(r)},\alpha_{N}^{(r)}\rangle=0\\
\langle\alpha_{\tilde{M}},\alpha_{\tilde{N_{\omega^{2}}}}\rangle=\sum_{r=o}^{m-1}\langle \alpha_{M}^{(r)},\alpha_{ N_{\omega^{2}}}^{(r)}\rangle=\sum_{r=o}^{m-1}\omega^{-2r}\langle\alpha_{M}^{(r)},\alpha_{N}^{(r)}\rangle=0\\
\vdots \\
\langle\alpha_{\tilde{M}},\alpha_{\tilde{N_{\omega^{(m-1)}}}}\rangle=\sum_{r=o}^{m-1}\langle\alpha_{M}^{(r)},\alpha_{N_{\omega^{(m-1)}}}^{(r)}\rangle=\sum_{r=o}^{m-1}\omega^{-(m-1)r}\langle\alpha_{M}^{(r)},\alpha_{N}^{(r)}\rangle=0
\end{gather*}
 
 Put $x_{r}=\langle\alpha_{M}^{(r)},\alpha_{N}^{(r)}\rangle$, then we get a system of equations which is expressed in a matrix form as,\\
$\underbrace{\begin{bmatrix}
1&1&\cdots&1\\
1&\omega^{-1} &\cdots&\omega^{-(m-1)}\\
\vdots\\
1&\omega^{-(m-1)}&\cdots&\omega^{-(m-1)^{2}}
\end{bmatrix} 
}_{\mathbf{P}}
\begin{bmatrix}
x_o\\
x_1\\
\vdots\\
x_{m-1}
\end{bmatrix} 
=
\begin{bmatrix}
0\\
0\\
\vdots\\
0
\end{bmatrix} $
 
 The matrix $\mathbf{P}$\ is a Vandermonde matrix with det($\mathbf{P})\neq0$\ i.e. $\mathbf{P}$ is invertible. So the above system has a unique solution which is trivial.
 
Therefore the solution to above system is, 
$$x_{r}=0, \ \mbox{ for } \ 0\leq r \leq m-1.$$
\begin{equation}
\langle\alpha_{M}^{(r)},\alpha_{N}^{(r)}\rangle=0,\ \mbox{ for } \ 0\leq r \leq m-1.
\end{equation}
\begin{definition}
 Let $M\in \mbox{Sim}(\mathcal{A}_{0})^{\Z/m\Z}$, the  twisted character $\tilde{\chi}_{M}: \mathcal{A}_{1}\rightarrow \mathcal{C}$ corresponding to M is defined as $\tilde{\chi}_{M}(a) = \chi_{\tilde{M}}(a)\ $ $\forall\ a\ \in \mathcal{A}_{1}.$
\end{definition}

\begin{definition}
 A linear functional $\gamma$ on \A$_{0}$-bimodule $E$ is said to a twisted class functional if $\gamma(am) = \gamma(ma)\ \forall\ m\in E , a \in \mathcal{A}_{o}$ and $cf_{A_{0}}(E)$ denotes the set of all twisted class functionals on $E$.
\end{definition}
We have thus proved
\begin{theorem}
Suppose \A\ is a $\Z/m\Z$-graded Frobenius $\star$-algebra.
Let $M,N \in \mbox{Sim}$(\A$_{0}$)$^{\Z/m\Z}$ then,
\begin{gather*}
\langle\tilde{\chi}_{M},\tilde{\chi}_{M}\rangle=\dfrac{\abs{c_{M}}^{2} \langle e_{\tilde{M}},e_{\tilde{M}}\rangle}{m}\\
\langle\tilde{\chi}_{M},\tilde{\chi}_{N}\rangle=0, \ \ \mbox{for} \ M\ncong N.
\end{gather*}
Moreover, the set of twisted characters $\{\tilde{\chi_{M}} : M\ \in\ Sim(A)^{\Z/m\Z}\}$ forms an orthogonal basis of the space of twisted class functionals on \A$_{1}$. 
\end{theorem}

Before proving the Theorem 4.11 we will prove a small lemma that gives us a more information about the space of twisted class functionals on $A_{1}$.

\begin{lemma}
$\phi^{-1}(cf_{\mathcal{A}_{0}} (\mathcal{A}_{1}))= Z_{\mathcal{A}_{0}}(\mathcal{A}_{-1})$.
\end{lemma}
\begin{proof}
Lemma follows from the the fact that \p\ is a \A -bimodule isomorphism between \A\ and \A$^{*}$. 
\end{proof}

Let $Z_{\mathcal{A}_{0}}(\mathcal{A})=\{a \in \mathcal{A} : ab=ba \ \forall\  b \in \mathcal{A}_{0}\}= \bigoplus_{r \in \Z/m\Z} Z_{\mathcal{A}_{0}}(\mathcal{A}_{r})\subseteq$ \A\ is a $\Z/m\Z$-graded Frobenious $\star$-algebra, where $\star$ and $\lambda_{0}$ are the restrictions of those for \A. 
The 0$^{th}$-grade of $Z_{\mathcal{A}_{0}}(\mathcal{A})$ is $Z(\mathcal{A}_{0})$ which is a commutative semisimple algebra. 
Hence every simple module over $Z(\mathcal{A}_{0})$ is one dimensional.
By lemma 4.3 we see that null socle of $E\otimes_{Z(\mathcal{A}_{0})}Z_{\mathcal{A}_{0}}(\mathcal{A})$ is zero for all $E \in $ Sim($Z(\mathcal{A}_{0})$).
So $E\overline{\otimes}Z_{\mathcal{A}_{0}}(\mathcal{A})= E\otimes_{Z(\mathcal{A}_{0})}Z_{\mathcal{A}_{0}}(\mathcal{A})$ for all $E \in $ Sim$(Z(\mathcal{A}_{0}))$.

As $Z(\mathcal{A}_{0})$ is a $\C$-vector subspace of $\mathcal{A}_{0}$ spanned by $\{ e_{M} : M \in Sim(\mathcal{A}_{0})\}$.
For every $M\in$ Sim$(\mathcal{A}_{0})$, $\C e_{M}$ is a simple $Z(\mathcal{A}_{0})$ module.
All simple $Z(\mathcal{A}_{0})$ are given by this way. 
We denote $\C e_{M}$ by $E_{M}$ for each $M \in$ Sim(\A$_{0}$).

Now we get a partial action of $\Z/m\Z$ on Sim$(Z(\mathcal{A}_{0}))$.
This action is given by for $r \in \Z/m\Z,\  E_{M}^{(r)} = E_{M}\otimes Z_{\mathcal{A}_{0}}(\mathcal{A}_{r})$, for all $E_{M}\in$ Sim$(Z(\mathcal{A}_{0}))$.
\begin{lemma}[see \cite{ta} Lemma 3.3]
For every $E_{M} \in$ Sim($Z(\mathcal{A}_{0})$), $E_{M} ^{(r)} = 0$ or $E_{M}^{(r)}\cong E_{M}$, 
Moreover $Z_{\mathcal{A}_{0}}(\mathcal{A}_{r})=\oplus_{E_{M} ^{(r)} =  E_{M}} e_{M}Z_{\mathcal{A}_{0}}(\mathcal{A}_{r}) \cong \oplus_{E_{M} = E_{M}^{(r)}} E_{M}$.
\end{lemma}

\begin{remark}
As a corollary to above lemma we get that dimension of $Z_{\mathcal{A}_{0}}(\mathcal{A}_{-1})$ is the cardinality of the set $Sim(Z(\mathcal{A}_{0}))^{\Z/m\Z}$.
\end{remark}

 \begin{lemma}
 There is a bijection between the sets Sim$(\mathcal{A}_{0}) ^{\Z/m\Z}$ and

 Sim$(Z_{\mathcal{A}_{0}}(\mathcal{A}_{0})) ^{\Z/m\Z}$.

 \end{lemma}
 
 \begin{proof}
 Let U $\in$ Sim$(\mathcal{A}_{0}) ^{\Z/m\Z}$, by using results from section 3 we get a simple \A -module $\tilde{U}$ such that $U \cong \tilde{U}$ as a representation of $\mathcal{A}_{0}$. 
 By Lemma 4.6 Hom$_{\mathcal{A}_{0}}(U, \tilde{U})$ is a simple $Z_{\mathcal{A}_{0}}(\mathcal{A})$-module.
 But the action of $Z_{\mathcal{A}_{0}}(\mathcal{A}_{0})$ on Hom$_{\mathcal{A}_{0}}(U, \tilde{U})$ is by a central character of $\mathcal{A}_{0}$, so Hom$_{\mathcal{A}_{0}}(U, \tilde{U})$ is isomorphic to Hom$_{\mathcal{A}_{0}}(U,U)$ as representation of $Z_{\mathcal{A}_{0}}(\mathcal{A}_{0})$.
 Therefore Hom$_{\mathcal{A}_{0}}(U, U)\in \mbox{Sim}(Z_{\mathcal{A}_{0}}(\mathcal{A}_{0}))^{\Z/m\Z}$.
 
 Now assume that Hom$_{\mathcal{A}_{0}}(U,U) \in \mbox{Sim}(Z_{\mathcal{A}_{0}}(\mathcal{A}_{0}))^{\Z/m\Z} i.e.$
  Hom$_{\mathcal{A}_{0}}(U,U)$ is a one dimensional representation of $Z_{\mathcal{A}_{0}}(\mathcal{A}_{0})$.
  By using results from section 3 and  Lemma 4.6, there exist a $M\in$Sim($\mathcal{A}$) such that  Hom$_{\mathcal{A}_{0}}(U,M)$ is  one dimensional representation of $Z_{\mathcal{A}_{0}}(\mathcal{A})$ and  Hom$_{\mathcal{A}_{0}}(U,M)\cong  \mbox{Hom}_{\mathcal{A}_{0}}(U,U)$ as a representation of $Z_{\mathcal{A}_{0}}(\mathcal{A}_{0}) $. 
  But the extension of representation of $Z_{\mathcal{A}_{0}}(\mathcal{A}_{0})$ in $\mbox{Sim}(Z_{\mathcal{A}_{0}}(\mathcal{A}_{0}))^{\Z/m\Z} $ to an irreducible representation of $Z_{\mathcal{A}_{0}}(\mathcal{A})$ has $m$-choices and all such representations of $Z_{\mathcal{A}_{0}}(\mathcal{A})$ are mutually non-isomorphic.
  So there are $m$ distinct simple \A -modules denote by $M_{1}, M_{2},\cdots,M_{m}$ such that 
  Hom$_{\mathcal{A}_{0}}(U,M_{r}) \cong \mbox{Hom}_{\mathcal{A}_{0}}(U,U)$ as a representation of $Z_{\mathcal{A}_{0}}(\mathcal{A}_{0})$ and
   Hom$_{\mathcal{A}_{0}}(U,M_{r}) \ncong \mbox{Hom}_{\mathcal{A}_{0}}(U,M_{s})$ as a representation of $Z_{\mathcal{A}_{0}}(\mathcal{A})$ for $r\neq s$.
  Therefore for each $r$, $M_{r} \cong U$ as a representation of $\mathcal{A}_{0}$, So there are $m$ extensions for $U$ as a representation of \A.
  Hence $U \in \mbox{Sim}(\mathcal{A}_{0}) ^{\Z/m\Z}$. 
 \end{proof}

\begin{corollary}
 Dimension of the space of twisted class functionals on $\mathcal{A}_{1}$ is a  cardinality of the set Sim$(\mathcal{A}_{0})^{\Z/m\Z} $.
\end{corollary}

\begin{proof}{\textbf{Proof of Theorem 4.11.}}

The first part of theorem is follows from (1) and (2) and the fact that $\phi(\alpha_{M}^{(-1)}) =\tilde{ \chi_{M}}.$ . 
The second part will follows from first part and corollary 4.16. 
\end{proof}

\section*{Acknowledgement}
I would like to thank Tanmay Deshpande for various comments, suggestions and useful discussions while working of this article.

\end{document}